\documentclass{article}
\usepackage{amssymb,amsmath,amsthm,graphicx,kotex}

\def\qed{\hfill {\hbox{${\vcenter{\vbox{               %HOLLOW SQUARE
   \hrule height 0.4pt\hbox{\vrule width 0.4pt height 6pt
   \kern5pt\vrule width 0.4pt}\hrule height 0.4pt}}}$}}}

\def\tr{\, \triangleright\, }
\def\utr{\, \underline{\tr}\, }
\def\otr{\, \overline{\tr}\, }

\newtheorem{theorem}{Theorem}
\newtheorem{lemma}[theorem]{Lemma}
\newtheorem{proposition}[theorem]{Proposition}
\newtheorem{corollary}[theorem]{Corollary}

\theoremstyle{definition}
\newtheorem{example}{Example}
\newtheorem{definition}{Definition}
\newtheorem{remark}{Remark}

\date{}

\title{\Large \textbf{Permutation Jones Polynomials}}

\author{Sam Nelson\footnote{Email: Sam.Nelson@cmc.edu. Partially supported by Simons Foundation collaboration grant 702597.}}

\begin{document}
\maketitle

\begin{abstract}
We introduce a generalization of the Jones polynomial for classical and 
virtual knots and links using colorings by a permutation $\sigma:X\to X$
of a finite set $X$. For $X=\{1\}$ and for classical knots, the invariant 
is equivalent to the usual Jones polynomial; for $X$ with cardinality greater 
than 1 the invariant expresses distinct information from the Jones polynomial 
for virtual knots and for classical and virtual links. We establish some 
properties of the new invariants and compute the polynomials for classical 
and virtual knots and links of small crossing number for a few small 
permutations. 
\end{abstract}

\parbox{6in} {\textsc{Keywords:} skein invariants, quantum invariants,
biquandle brackets, virtual knots and links

\smallskip

 }

\section{Introduction}\label{I}

The classical skein invariants such as the Jones/Kauffman bracket polynomial, 
the Alexander/Conway polynomial and the HOMFLYpt polynomial have each been
the starting point for many interesting and fruitful programs of enhancement
and generalization leading to new invariants of knots and links. Examples 
include the twisted Alexander polynomials \cite{W}, multivariable Alexander 
polynomials \cite{M}, colored Jones polynomials \cite{MM}, Khovanov homology 
\cite{Khom} and Knot Floer homology \cite{OS} for classical knots and links 
as well as the Miyazawa polynomial \cite{MP}, the Arrow polynomial \cite{KA}, 
the Sawollek polynomial \cite{KR} for virtual knots and links and more. 
Indeed, the study of the Jones polynomial for classical knots was one of the 
original motivations for the introduction of virtual knots \cite{KV}.

Originally introduced in \cite{NOSY}, \textit{biquandle virtual brackets} are 
skein invariants of biquandle-colored oriented classical and virtual knots 
and links defined using a 3-term skein relation with coefficients which
are functions of the biquandle colors, motivated by the idea that a virtual 
crossing is really a smoothing rather than a crossing. 

While working on examples for another paper \cite{NOO2}, the author came 
across a family of polynomial biquandle virtual brackets over a type 
of biquandle known as \textit{constant action biquandles} whose upper and 
lower actions are given by the same permutation for each element. These 
invariants can be defined directly via skein relations for knots and links 
with colorings by permutations $\sigma:X\to X$ without 
explicit reference to biquandles and have some interesting properties: for 
classical knots they reduce to the standard Jones polynomial, but for 
classical links they can give extra information not present in the standard
Jones polynomial and for virtual knots and links they can distinguish 
some Jones-equivalent cases while failing to distinguish some cases that are 
distinguished by the Jones polynomial.

The paper is organized as follows. In Section \ref{PJP} we define the
invariants via skein relations, provide an elementary proof of invariance 
and establish some basic properties of these invariants.
In Section \ref{CE} we show examples of computing the invariant 
and provide tables of invariant values for various sets of classical and 
virtual knots and links. We conclude with open questions in Section \ref{Q}.

This paper, including all text, diagrams and computational code,
was produced strictly by the author without the use of generative AI in any 
form.

\section{Permutation Jones Polynomials}\label{PJP}

We begin this section with a definition.

\begin{definition}
Let $X=\{1,2,\dots,n\}$, let $\sigma:X\to X$ be a bijection and $D$ a diagram
of an oriented classical or virtual knot or link. Then a 
\textit{$\sigma$-coloring} of $D$ is an assignment of an element of $X$ to each
semiarc in $D$ (i.e., each edge in the graph obtained by regarding crossing 
points as vertices) such that at each crossing we have the following
picture:
\[\includegraphics{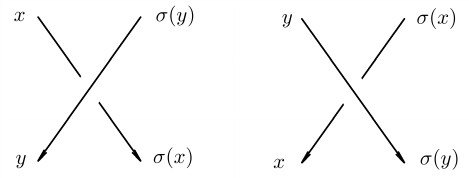}.\]
That is, going under a crossing point picks up $\sigma^{\mathrm{writhe}}$ and
going over a crossing point picks up $\sigma^{-\mathrm{writhe}}$.
\end{definition}

We will denote a permutation $\sigma:\{1,2,\dots, n\}\to \{1,2,\dots, n\}$
by the vector of its images, i.e., we will write 
\[\sigma=[\sigma(1),\sigma(2),\dots,\sigma(n)].\]

\begin{theorem}
The set $\mathcal{C}_{\sigma}(L)$ of $\sigma$-colorings of a classical or 
virtual oriented knot or link $L$ is invariant in the sense that for any 
two diagrams $D,D'$ related by Reidemeister moves there is a bijection 
$f:\mathcal{C}_{\sigma}(D)\to\mathcal{C}_{\sigma}(D')$ such that
each $\sigma$-coloring $D_j$ of $D$ is related to its image $f(D_j)$
by $\sigma$-colored Reidemeister moves. 
\end{theorem}

\begin{proof}
This follows easily from the 
fact that $\sigma$-colorings are biquandle colorings by the constant action
biquandle structure on $X$ defined by $x\utr y=x\otr y=\sigma(x)$. The 
interested reader can also check the Reidemeister moves directly.
\end{proof}

\begin{example}\label{ex:1}
The Hopf link has four $[2,1]$-colorings as shown.
\[\includegraphics{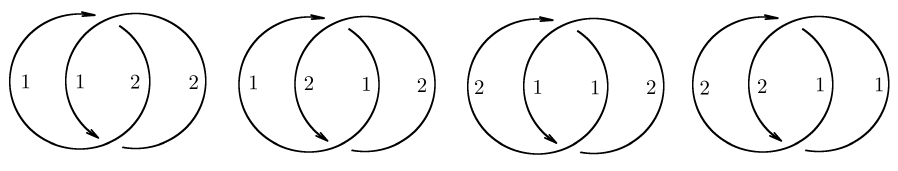}\]
\end{example}

\begin{remark}
For some virtual knots and links $L$ and permutations $\sigma:X\to X$,
$\mathcal{C}_{\sigma}(L)$ may be empty. For example, there are no 
$[2,1]$-colorings of the 
\textit{virtual Hopf link}
\[\includegraphics{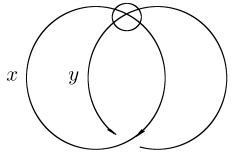}\]
since such a coloring would require $x=\sigma(x)$ and $y=\sigma(y)$ but
$\sigma=[2,1]$ has no fixed points.
\end{remark}

We will now introduce a 3-variable polynomial invariant of
classical and virtual knots and links associated to a permutation
$\sigma:X\to X$. For each $\sigma$-coloring of a diagram $D$, we
compute the state-sum using the following skein relations:
at \textit{monochromatic} crossings, i.e., crossings where the colors
on the left strands are equal, we have a Kauffman bracket-style
skein relation, 
\[\includegraphics{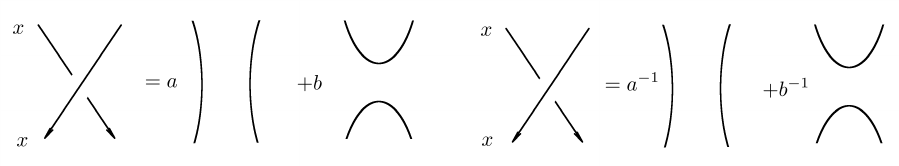}\]
while at \textit{polychromatic} crossings, i.e.,
crossings with unequal colors on the left, we virtualize the crossing
and multiply by a coefficient of $v^{\mathrm{crossing \ sign}}$. 
\[\includegraphics{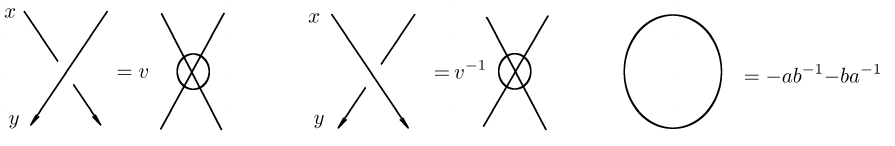}\]

Each component circle in a Kauffman state (which may have virtual crossings) 
contributes a factor of $-ab^{-1}-ba^{-1}$. We note that the value of Kauffman
state in unchanged by pure virtual Reidemeister moves.
We normalize the state-sum to the writhe-zero case by multiplying 
by $(-a^2b^{-1})^{-\mathrm{writhe}(L)}$ 
\[\includegraphics{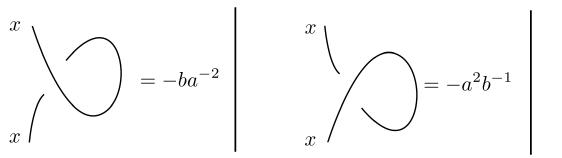}\]
and sum over the set of $\sigma$-colorings.
That is, we have:

\begin{definition}
Let $X=\{1,2,\dots, n\}$ be a finite set and $\sigma:X\to X$ a bijection.
Then for any classical or virtual knot or link $L$ represented by a diagram
$D$ we define the \textit{permutation Jones polynomial} $J_{\sigma}(K)$ to be
\[\sum_{d\in \mathcal{C}_{\sigma}(L)} (-a^2b^{-1})^{-\mathrm{writhe}(D)} 
\left(\sum_{s\in \mathcal{S}(d)}a^xb^yv^z(-ab^{-1}-ba^{-1})^{|s|}\right) 
\]
where $|s|$ is the number of components in the fully-smoothed 
Kauffman state (which may have virtual crossings) $s$,
$\mathcal{S}(d)$ is the set of Kauffman states of the $\sigma$-colored diagram
$d$ and $a^xb^yv^z$ is the monomial of skein coefficients associated to 
the state $s$. 
\end{definition}

We next come to our main result.

\begin{theorem}
The polynomial $J_{\sigma}$ is an invariant of classical and virtual knots and
links for every permutation $\sigma:X\to X$.
\end{theorem}

\begin{proof}
The quickest proof is to observe that $J_{\sigma}$ is a biquandle virtual 
bracket; hence the multiset of state-sum values over of the set of 
colorings is an invariant. It then follows immediately that the sum of 
the elements of the multiset is also an invariant. 

For those unfamiliar with biquandle brackets and biquandle virtual brackets, 
we present an elementary proof of the invariance of $J_{\sigma}$ independent 
of the biquandle virtual bracket construction. It suffices to show that 
the state-sum contributions are equal on both sides of each move in a 
generating set of oriented $\sigma$-colored Reidemeister moves for each 
possible coloring type. One such set consists of all four oriented type I 
moves, all four oriented type II moves and the all-positive type III move 
as well as the mixed virtual move; the purely virtual moves do not change the 
state sum value.

First, let us consider the direct Reidemeister II move. There are two possible
cases -- either both crossings are monochromatic 
\[\includegraphics{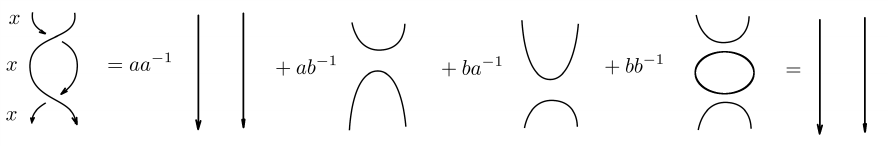}\]
or both are polychromatic
\[\includegraphics{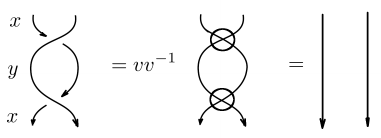}.\]
We have depicted one of the two possible moves; the other is similar.

Next, let us consider the reverse Reidemeister II move. Again there are two 
possible cases -- either both crossings are monochromatic
\[\includegraphics{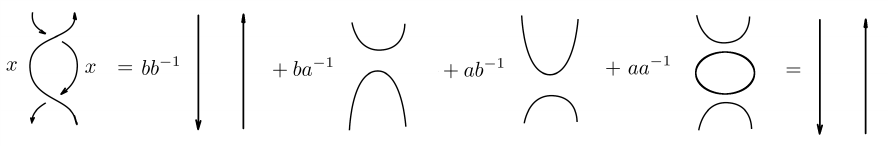}\]
or both are polychromatic
\[\includegraphics{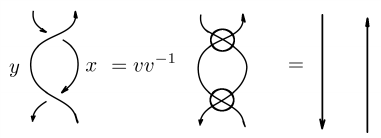}.\]
We have again depicted one of the two possible moves; the other is similar.

Next let us consider the all-positive Reidemeister III move. The 
all-monochromatic case is the same as for the classical Kauffman bracket; 
let us consider the the mixed cases. We have
\[\includegraphics{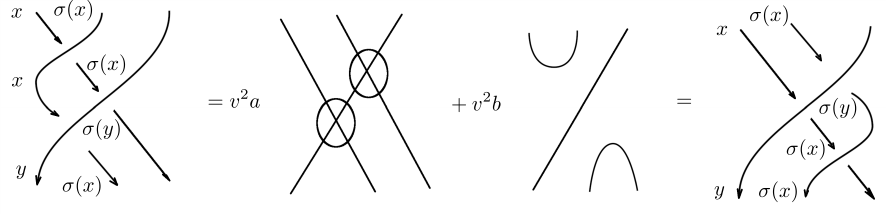},\]
\[\includegraphics{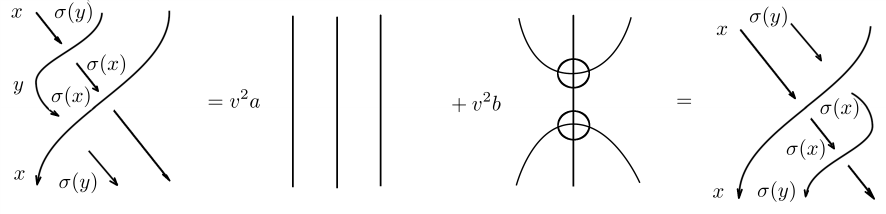}\]
and
\[\includegraphics{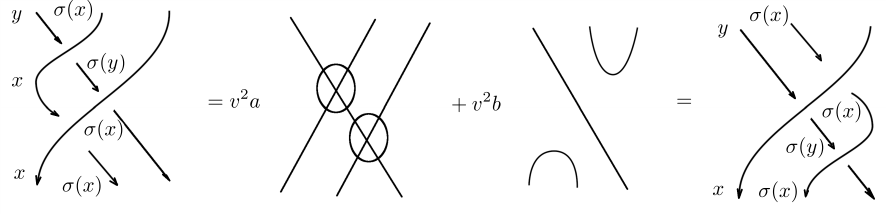}.\]

Next, we note that the mixed virtual move does not change the 
state-sum value, either in the monochromatic case
\[\includegraphics{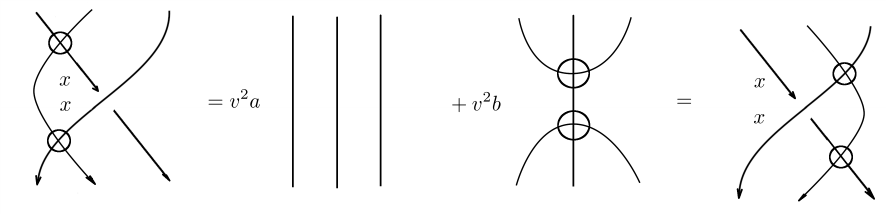}\]
or in the polychromatic case
\[\includegraphics{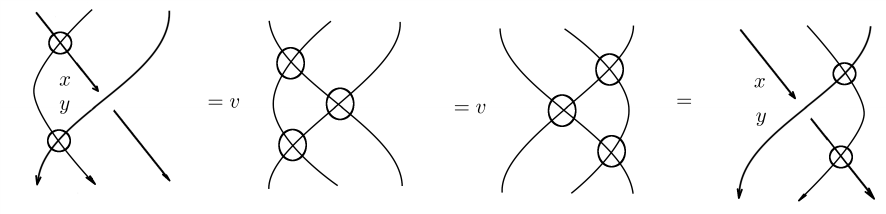}.\]

Finally, as in the case of the Kauffman bracket, the Reidemeister I move
does change the state-sum value in a predictable way; hence we normalize by
adjusting to the writhe-zero case. We note that Reidemeister I crossings
are always monochromatic; we depict two of the four cases.
\[\includegraphics{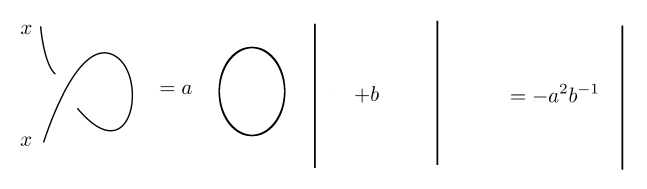}\]
\[\includegraphics{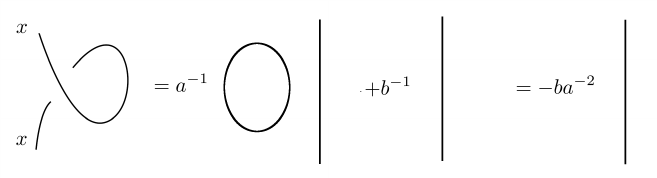}.\]

\end{proof}

\begin{corollary}\label{cor:1}
The total degree of each monomial in $J_{\sigma}(L)$ is always zero.
\end{corollary}

\begin{proof}
Each monomial comes from a product of factors $-a^{-1}b$ or $-ab^{-1}$ from
Kauffman states times a skein coefficient product of total degree equal to the
writhe of the diagram; the normalization factor $(-a^2b^{-1})^{-\mathrm{writhe}(d)}$
then subtracts exactly the writhe of the diagram from the total degree for each
monomial.
\end{proof}

In light of Corollary \ref{cor:1}, we could reduce the number of variables
in $J_{\sigma}$ to two by, for example, setting $b=a^{-1}$. However, the $v$ 
variable sometimes replaces $a$ and sometimes replaces $b$, and hence it
seems preferable to keep all three.

The classical Kauffman bracket/Jones polynomial can be obtained from $J_{[1]}$
by setting $b=a^{-1}$ and dividing out a factor of $-a^2-a^{-2}$ to set the 
unknot value to 1. In the case of classical knots, we have the following lemma:

\begin{lemma} \label{lem:1}
In any $\sigma$-coloring of a classical knot, all crossings are monochromatic.
\end{lemma}

\begin{proof}
Consider a crossing in a $\sigma$-coloring of a classical knot. Let us call 
the undercrossing label $x$; then as we travel around the loop counterclockwise
to return to the crossing, we encounter a series of crossings, some positive 
and some negative, some going over and some going under. 
\[\includegraphics{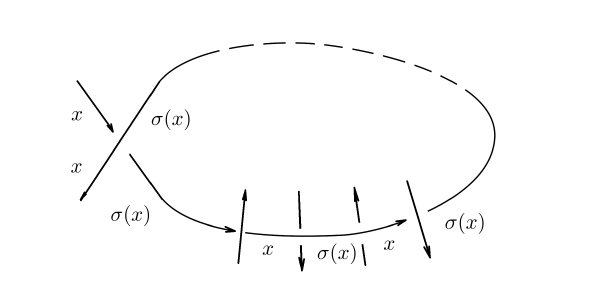}\]
Examination of the 
$\sigma$-coloring rule shows that strands crossing into the disk decrease 
the exponent of $\sigma$ on the boundary strand by 1 while strands crossing 
out of the disk increase the power of $\sigma$ on the boundary strand by 1.
Without virtual crossings, each crossing oriented into the planar disk 
bounded by the loop must be matched by one oriented out; hence when we return 
to the crossing the color is $\sigma^0(x)=x$ and the crossing is monochromatic. 
\end{proof}

We then have the following:

\begin{corollary}
For a classical knot $K$, $J_{\sigma}(K)=|\mathcal{C}_{\sigma}(K)|J_{[1]}(k)$.
\end{corollary}

We note that classical links can have both monochromatic and polychromatic 
crossings as seen in Example \ref{ex:1}; hence for classical links $L$, 
$J_{\sigma}(L)$ can carry more information than $J_{[1]}$. In particular we have:

\begin{corollary}
Let $L=L_1\cup L_2$ be a classical link of two components $L_1$ and $L_2$.
If any term of $J_{\sigma}(L)$ includes a nonzero power of $v$, then
$L$ is not a split link.
\end{corollary}

\begin{proof}
If $L=L_1\cup L_2$ is split then each $\sigma$-coloring of $L$ contributes
a product $J_{\sigma}(L_1)J_{\sigma}(L_2)$ of $J_{\sigma}$-values of classical
knots, both of which have only monochromatic crossings and hence cannot have
nonzero powers of $v$.
\end{proof}

There are known examples of split and non-split classical links sharing
the same value of the classical Jones polynomial, e.g. Thistlethwaite's 
examples of non-split links with trivial Jones polynomial in \cite{T}.
It follows that the Jones polynomial alone does not determine whether
a link in split.

For virtual knots and virtual links, the differences between the permutation 
Jones polynomials for nontrivial permutations $\sigma$ are even more clear. 
From Lemma \ref{lem:1} we have:

\begin{corollary}
If a virtual knot $K$ has $J_{\sigma}(K)$-value with a nonzero power of $v$ for
any $\sigma$, then $K$ is not classical.
\end{corollary}

\begin{proof}
This follows immediately from Lemma \ref{lem:1}.
\end{proof}

We end this section with a few observations about the permutation Jones 
polynomials on virtual knots and links.

\begin{proposition}
As with the classical Jones polynomial, $J_{\sigma}$ detects (vertical)
mirror images with $J_{\sigma}(\bar{L})$ obtained from $J_{\sigma}(L)$
by negating all exponents.
\end{proposition}

\begin{proof}
It suffices to observe that simultaneously mirroring all crossings reverses
all of the crossing signs which inverts all of the skein coefficients and
negates the diagram's writhe while preserving the value of each Kauffman state 
component.
\end{proof}

\begin{proposition}\label{prop:mut}
As with the classical Jones polynomial, $J_{\sigma}$ is invariant under mutation.
\end{proposition}

\begin{proof}
It suffices to observe that for each $\sigma$-coloring, mutation does not 
change the mono/polychomatic status of a crossing; hence as in the classical 
Jones polynomial, Kauffman states match up one-to-one with equal skein 
coefficients before and after mutation for each $\sigma$-coloring
\end{proof}

\begin{proposition}\label{prop:2}
For permutations $\sigma:X\to X$ with $|X|>1$, $J_{\sigma}$ is not generally
invariant under Kauffman virtualization (i.e., reversing an arrow in a Gauss 
diagram while fixing the crossing sign, or equivalently, flanking a classical 
crossing with virtual crossings).
\end{proposition}

\begin{proof}
Consider the case of the virtual knot $3.1$ 
\[\includegraphics{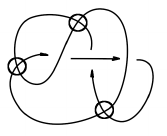}\]
with the permutation 
$\sigma=[2,3,1]$. We compute that $J_{[2,3,1]}(3.1)=3a^3b^{-2}v^{-1}+3av^{-1}$
while Kauffman-virtualizing the leftmost crossing yields the unknot 
\[\includegraphics{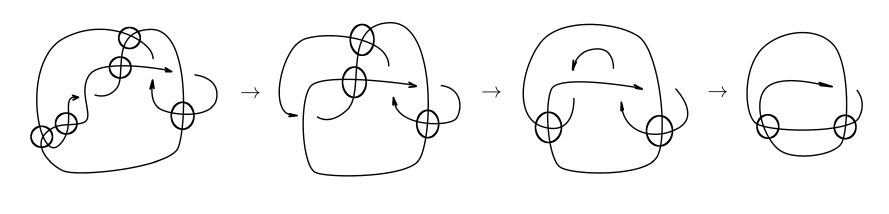}\]
with
$J_{[2,3,1]}(K')=-3ab^{-1}-3ba^{-1}$.
%>>> permskein([[-1.5, 2, 3.5, -2, 1.5, -3.5]],[2,3,1])
%3*a**3/(b**2*v) + 3*a/v
%>>> permskein([[1.5, 2, -3.5, -2, -1.5, 3.5]],[2,3,1])
%-3*a/b - 3*b/a
\end{proof}

\begin{proposition}
For general permutations $\sigma$, $J_{\sigma}$ is not determined by the Jones
polynomial.
\end{proposition}

\begin{proof}
In fact the same example from the proof of proposition \ref{prop:2} 
establishes this result as well since both $3.1$ and its virtualization 
the unknot have trivial Jones polynomial. Other examples are plentiful.
\end{proof}

\begin{proposition}
For general permutations $\sigma$, $J_{\sigma}$ does not determine the Jones
polynomial.
\end{proposition}

\begin{proof}
The virtual knot $4.2$ 
\[\includegraphics{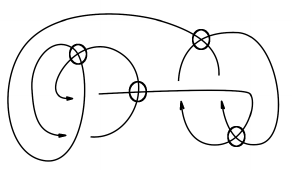}\]
has trivial $J_{[2,3,1]}$-value of $-3ab^{-1}-3ba^{-1}$
but nontrival Jones polynomial value 
\[J_{[1]}(4.2)=a^4b^{-4}+a^3b^{-3}-2ab^{-1}-2-2ba^{-1}+b^3a^{-3}+b^4a^{-4}.\]
\end{proof}

\section{Computations and Examples}\label{CE}

In this section we collect some example computations and results.

\begin{example}\label{ex:2}
Let us illustrate the computation of $J_{[2,1]}$ for the Hopf link $L2a1$. 
In Example \ref{ex:1} we showed the four $[2,1]$-colorings of $L2a1$. Starting 
with the first coloring and applying the skein relations, we have
\[\includegraphics{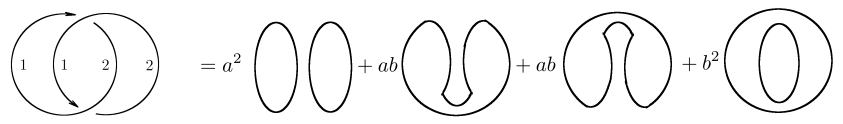}\]
while the second coloring yields
\[\includegraphics{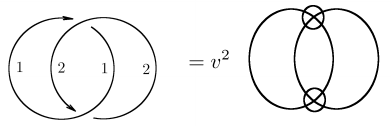}.\]
The third and fourth colorings yield the same results as the second and first
respectively, so summing them and multiplying by the writhe adjustment factor
of $(-a^2b^{-1})^{-2}$ we have
\begin{eqnarray*}
J_{[2,1]}(L2a1) 
& = & (-a^2b^{-1})^{-2}(2a^2(-ab^{-1}-ba^{-1})^2+4ab(-ab^{-1}-ba^{-1})+2b^2(-ab^{-1}-ba^{-1})^2 \\ & & \quad+2v^2(-ab^{-1}-ba^{-1})^2) \\
& = & (b^2a^{-4})(2a^2(a^2b^{-2}+2+b^2a^{-2})+4ab(-ab^{-1}-ba^{-1})+2b^2(a^2b^{-2}+2+b^2a^{-2}) \\ & & \quad +2v^2(a^2b^{-2}+2+b^2a^{-2})) \\
& = & (b^2a^{-4})(
2a^4b^{-2}+4a^2+2b^2
-4a^2-4b^2
+2a^2+4b^2+2b^4a^{-2} \\ & & \quad
+2v^2a^2b^{-2}+4v^2+2v^2b^2a^{-2}) \\
%& = & 
%2+2a^{-4}b^4
%+2a^{-2}b^2+2a^{-6}b^6 
%+2v^2a^{-2}+4v^2a^{-4}b^2+2v^2a^{-6}b^4 \\
& = & 2+2b^2a^{-2}+2v^2a^{-2}+2b^4a^{-4}+4b^2v^2a^{-4}+2b^6a^{-6}+2b^4v^2a^{-6}.
\end{eqnarray*}
The fact that the terms contain nonzero powers of the $v$ variable
shows that this two-component classical link is nonsplit.
\end{example}

\begin{example}
We list the $J_{\sigma}$-values for the prime classical links with up to
7 crossings for $\sigma=[2,1]$ as computed by our \texttt{python} code
in the table.
\[\begin{array}{r|l}
L & J_{[2,1]}(L) \\ \hline
 & \\
L2a1 & 2 + 2b^2a^{-2} + 2v^2a^{-2} + 2b^4a^{-4} + 4b^2v^2a^{-4} + 2b^6a^{-6} + 2b^4v^2a^{-6}\\ 
L4a1 & 2 + 2b^6a^{-6} + 2b^2v^4a^{-6} + 2b^8a^{-8} + 4b^4v^4a^{-8} + 2b^{10}a^{-10} + 2b^6v^4a^{-10}\\ 
L5a1 & -2a^8b^{-8} + 2a^6b^{-6} + 2a^4b^{-4} + 4a^2b^{-2} + 6 + 2b^2a^{-2} + 2b^4a^{-4}\\ 
L6a1 & 2a^4b^{-4} - 2a^2b^{-2} + 2b^4a^{-4} + 4b^6a^{-6} + 2b^2v^4a^{-6} + 4b^4v^4a^{-8} + 2b^{10}a^{-10} + 2b^6v^4a^{-10}\\ 
L6a2 & 2b^2a^{-2} + 2b^6a^{-6} + 2b^4v^6a^{-10} + 2b^{12}a^{-12} + 4b^6v^6a^{-12} + 2b^8v^6a^{-14} + 2b^{16}a^{-16}\\ 
L6a3 & 2b^4a^{-4} + 2b^6a^{-6} + 2b^8a^{-8} + 2b^4v^6a^{-10} + 4b^6v^6a^{-12} + 2b^8v^6a^{-14} + 2b^{18}a^{-18}\\ 
L6a4 & 2a^7b^{-7} - 4a^5b^{-5} - 8a^3b^{-3} - 22ab^{-1} - 22ba^{-1} - 8b^3a^{-3} - 4b^5a^{-5} + 2b^7a^{-7}\\ 
L6a5 & -2ba^{-1} + 2b^3a^{-3} - 2b^5a^{-5} - 4b^7a^{-7} - 6b^3v^4a^{-7} - 4b^9a^{-9} - 12b^5v^4a^{-9} - 4b^{11}a^{-11} \\ 
& \ - 12b^7v^4a^{-11} - 12b^9v^4a^{-13} - 2b^{15}a^{-15} - 6b^{11}v^4a^{-15}\\ 
L6n1 & -2b^3a^{-3} - 2b^5a^{-5} - 2b^7a^{-7} - 6b^3v^4a^{-7} - 2b^9a^{-9} - 12b^5v^4a^{-9} - 4b^{11}a^{-11} \\ 
&\  - 12b^7v^4a^{-11} - 4b^{13}a^{-13} - 12b^9v^4a^{-13} - 6b^{11}v^4a^{-15}\\ 
L7a1 & 2a^{10}b^{-10} - 4a^8b^{-8} + 2a^6b^{-6} + 4a^2b^{-2} + 8 + 2b^2a^{-2} + 4b^4a^{-4} - 2b^6a^{-6}\\ 
L7a2 & 2b^2a^{-2} - 2b^4a^{-4} + 4b^6a^{-6} + 2b^8a^{-8} + 2b^4v^4a^{-8} + 2b^{10}a^{-10} + 4b^6v^4a^{-10} + 2b^{12}a^{-12} \\ 
& \ + 4b^8v^4a^{-12} - 2b^{14}a^{-14} + 2b^{10}v^4a^{-14} + 2b^{16}a^{-16} - 2b^{12}v^4a^{-16} - 2b^{18}a^{-18} - 2b^{14}v^4a^{-18}\\ 
L7a3 & -2a^{14}b^{-14} \\ 
& \ + 2a^{12}b^{-12} - 4a^{10}b^{-10} - 2a^8b^{-8} + 4a^6b^{-6} + 6a^4b^{-4} + 8a^2b^{-2} + 2 + 2b^2a^{-2}\\ 
L7a4 & -2a^{12}b^{-12} + 2a^{10}b^{-10} + 2a^6b^{-6} + 2a^2b^{-2} + 8 + 2b^2a^{-2} + 2b^4a^{-4}\\ 
L7a5 & 2a^4b^{-4} - 2a^2b^{-2} + 2 + 2v^2a^{-2} + 2b^4a^{-4} + 4b^2v^2a^{-4} + 4b^6a^{-6} + 2b^4v^2a^{-6} + 2b^{10}a^{-10} - 2b^{12}a^{-12}\\ 
L7a6 & -2a^8b^{-8} + 2a^6b^{-6} + 2a^2b^{-2} + 2 + 2v^2a^{-2} + 2b^4a^{-4} + 4b^2v^2a^{-4} + 2b^4v^2a^{-6} + 2b^8a^{-8}\\ 
L7a7 & 2a^7b^{-7} - 4a^5b^{-5} - 6ab^{-1} - 10ba^{-1} - 2v^4a^{-1}b^{-3} - 10b^3a^{-3} - 4v^4a^{-3}b^{-1} - 14b^5a^{-5} - 4bv^4a^{-5} \\ 
& \ - 4b^7a^{-7} - 4b^3v^4a^{-7} - 2b^9a^{-9} - 2b^5v^4a^{-9}\\ 
L7n1 & 2a^8b^{-4}v^{-4} + 2a^8b^{-8} + 4a^6b^{-2}v^{-4} + 2a^6b^{-6} + 4a^4v^{-4} + 2a^4b^{-4} + 2a^2b^2v^{-4} \\ 
& \ + 2a^2b^{-2} - 2b^4v^{-4} + 2 - 2b^6a^{-2}v^{-4} - 2b^4a^{-4}\\ 
L7n2 & 6 + 6b^2a^{-2} + 6b^4a^{-4} + 4b^6a^{-6} - 2b^8a^{-8} - 2b^{10}a^{-10} - 2b^{12}a^{-12}
\end{array}
\]
\end{example}

\begin{example}
We list the $J_{\sigma}$-values for the prime virtual knots with up to
4 classical crossings for $\sigma=[2,1]$  and $\sigma=[2,3,1]$ 
as computed by our \texttt{python} 
code in the tables.
\[\begin{array}{r|l}
J_{[2,1]}(K) & K \\ \hline
& \\
-2bv^2a^{-3} - 2b^3v^2a^{-5} & 2.1,\ 4.28,\ 4.84,\ 4.88,\ 4.104\\
-2ab^{-1} - 2ba^{-1} & 3.1,\ 3.7,\ 4.2,\ 4.6,\ 4.8,\ 4.12,\ 4.13,\ 4.17,\ 4.19,\ 4.26,\\ & 4.32,\ 4.35,\ 4.42,\ 4.46,\ 4.47,\ 4.51,\ 4.55,\ 4.56,\ 4.58,\\ & 4.59,\ 4.66,\ 4.67,\ 4.71,\ 4.72,\ 4.75,\ 4.76,\ 4.77, 4.85,\\ & 4.93,\ 4.96, 4.97,\ 4.98,\ 4.102,\ 4.103,\ 4.106,\ 4.107  \\
-2a^5b^{-3}v^{-2} - 2a^3b^{-1}v^{-2} & 3.2,\ 3.3,\ 3.4,\ 4.4,\ 4.5,\ 4.11,\ 4.18,\ 4.27,\ 4.30,\ 4.33,\\ & 4.38,\ 4.39,\ 4.44,\ 4.45,\ 4.49,\ 4.54,\ 4.62,\ 4.63,\ 4.74,\\ & 4.81,\ 4.82,\ 4.83,\ 4.87,\ 4.92,\ 4.94,\ 4.95,\ 4.101 \\
2a^9b^{-9} - 2a^5b^{-5} - 2a^3b^{-3} - 2ab^{-1} & 3.5,\ 4.89,\ 4.105 \\
-2ba^{-1} - 2b^3a^{-3} - 2b^5a^{-5} + 2b^9a^{-9} & 3.6 \\
-2a^9b^{-5}v^{-4} - 2a^7b^{-3}v^{-4} & 4.1,\ 4.3,\ 4.7,\ 4.25,\ 4.43,\ 4.53,\ 4.73,\ 4.80,\ 4.91,\\ & 4.100 \\
2a^{10}b^{-8}v^{-2} - 2a^7b^{-5}v^{-2} - 2a^6b^{-4}v^{-2} - 2a^5b^{-3}v^{-2} & 4.9,\ 4.15,\ 4.29,\ 4.37,\ 4.48,\ 4.61,\ 4.69,\ 4.78\\ 
2a^6b^{-6} - 2a^3b^{-3} - 2a^2b^{-2} - 2ab^{-1} & 4.10,\ 4.16,\ 4.23,\ 4.31,\ 4.41,\ 4.50,\ 4.57,\ 4.65,\ 4.70,\\ &  4.79 \\
-2a^3b^{-1}v^{-2} - 2a^2v^{-2} - 2abv^{-2} + 2b^4a^{-2}v^{-2} & 4.14,\ 4.20,\ 4.22,\ 4.34,\ 4.40,\ 4.52,\ 4.60,\ 4.64 \\
-2ba^{-1} - 2b^2a^{-2} - 2b^3a^{-3} + 2b^6a^{-6} & 4.21,\ 4.24,\ 4.36,\ 4.68 \\
-2a^5b^{-5} - 2b^5a^{-5} & 4.86,\ 4.90,\ 4.99,\ 4.108 \\
\end{array}
%[[-2*b*v**2/a**3 - 2*b**3*v**2/a**5.[(2,1), (4, 28), (4, 84), (4, 88), (4, 104)]], [-2*a/b - 2*b/a, [(3, 1), (3, 7), (4, 2), (4, 6), (4, 8), (4, 12), (4, 13), (4, 17), (4, 19), (4, 26), (4, 32), (4, 35), (4, 42), (4, 46), (4, 47), (4, 51), (4, 55), (4, 56), (4, 58), (4, 59), (4, 66), (4, 67), (4, 71), (4, 72), (4, 75), (4, 76), (4, 77), (4, 85), (4, 93), (4, 96), (4, 97), (4, 98), (4, 102), (4, 103), (4, 106), (4, 107)]], [-2*a**5/(b**3*v**2) - 2*a**3/(b*v**2), [(3, 2), (3, 3), (3, 4), (4, 4), (4, 5), (4, 11), (4, 18), (4, 27), (4, 30), (4, 33), (4, 38), (4, 39), (4, 44), (4, 45), (4, 49), (4, 54), (4, 62), (4, 63), (4, 74), (4, 81), (4, 82), (4, 83), (4, 87), (4, 92), (4, 94), (4, 95), (4, 101)]], [2*a**9/b**9 - 2*a**5/b**5 - 2*a**3/b**3 - 2*a/b, [(3, 5), (4, 89), (4, 105)]], [-2*b/a - 2*b**3/a**3 - 2*b**5/a**5 + 2*b**9/a**9, [(3, 6)]], [-2*a**9/(b**5*v**4) - 2*a**7/(b**3*v**4), [(4, 1), (4, 3), (4, 7), (4, 25), (4, 43), (4, 53), (4, 73), (4, 80), (4, 91), (4, 100)]], [2*a**10/(b**8*v**2) - 2*a**7/(b**5*v**2) - 2*a**6/(b**4*v**2) - 2*a**5/(b**3*v**2), [(4, 9), (4, 15), (4, 29), (4, 37), (4, 48), (4, 61), (4, 69), (4, 78)]], [2*a**6/b**6 - 2*a**3/b**3 - 2*a**2/b**2 - 2*a/b, [(4, 10), (4, 16), (4, 23), (4, 31), (4, 41), (4, 50), (4, 57), (4, 65), (4, 70), (4, 79)]], [-2*a**3/(b*v**2) - 2*a**2/v**2 - 2*a*b/v**2 + 2*b**4/(a**2*v**2), [(4, 14), (4, 20), (4, 22), (4, 34), (4, 40), (4, 52), (4, 60), (4, 64)]], [-2*b/a - 2*b**2/a**2 - 2*b**3/a**3 + 2*b**6/a**6, [(4, 21), (4, 24), (4, 36), (4, 68)]], [-2*a**5/b**5 - 2*b**5/a**5, [(4, 86), (4, 90), (4, 99), (4, 108)]]]
\]

\[\begin{array}{r|l}
K & J_{[2,3,1]}(K) \\ \hline
& \\
-3bv^2a^{-3} - 3b^3v^2a^{-5} & 2.1,\ 4.21,\ 4.36 \\
3a^3b^{-2}v^{-1} + 3a/v & 3.1,\ 3.4,\ 4.10,\ 4.17,\ 4.19,\ 4.20,\ 4.23,\ 4.32,\\ & 4.34,\ 4.35,\ 4.38,\ 4.45,\ 4.47,\ 4.49,\ 4.50,\ 4.57,\\ & 4.70,\ 4.83,\ 4.88,\ 4.97,\ 4.103\\
-3a^5b^{-3}v^{-2} - 3a^3b^{-1}v^{-2} & 3.2,\ 3.5,\ 3.7,\ 4.4,\ 4.5,\ 4.18,\ 4.27,\ 4.30,\ 4.33,\\ & 4.44,\ 4.54,\ 4.62,\ 4.65,\ 4.74,\ 4.79,\ 4.85,\ 4.86,\\ & 4.94,\ 4.96,\ 4.106\\
3a^7b^{-4}v^{-3} + 3a^5b^{-2}v^{-3}  & 3.3,\ 4.11,\ 4.15,\ 4.29,\ 4.63,\ 4.80,\ 4.81,\ 4.87,\\ & 4.93\\
-3ba^{-1} - 3b^3a^{-3} - 3b^5a^{-5} + 3b^9a^{-9} & 3.6 \\
-3a^9b^{-5}v^{-4} - 3a^7b^{-3}v^{-4} & 4.1,\ 4.3,\ 4.7,\ 4.25,\ 4.37,\ 4.43,\ 4.48,\ 4.53,\\ & 4.73,\ 4.82,\ 4.89,\ 4.100 \\
-3ab^{-1} - 3ba^{-1} & 4.2,\ 4.6,\ 4.8,\ 4.12,\ 4.13,\ 4.14,\ 4.22,\ 4.46,\\ & 4.51,\ 4.55,\ 4.56,\ 4.58,\ 4.59,\ 4.64,\ 4.66,\ 4.71,\\ & 4.72,\ 4.75,\ 4.76,\ 4.77,\ 4.84,\ 4.90,\ 4.98,\ 4.107 \\
3a^{10}b^{-8}v^{-2} - 3a^7b^{-5}v^{-2} - 3a^6b^{-4}v^{-2} - 3a^5b^{-3}v^{-2} & 4.9,\ 4.61,\ 4.69,\ 4.78,\ 4.91 \\
3a^6b^{-6} - 3a^3b^{-3} - 3a^2b^{-2} - 3ab^{-1} & 4.16,\ 4.31,\ 4.39,\ 4.41,\ 4.95,\ 4.101\\
-3ba^{-1} - 3b^2a^{-2} - 3b^3a^{-3} + 3b^6a^{-6} & 4.24,\ 4.68\\
3va^{-1} + 3b^2va^{-3} & 4.26,\ 4.28,\ 4.67\\
-3a^3b^{-1}v^{-2} - 3a^2v^{-2} - 3abv^{-2} + 3b^4a^{-2}v^{-2} & 4.40,\ 4.42,\ 4.52,\ 4.60,\ 4.102\\
3a^{10}b^{-10} + 3a^9b^{-9} - 3a^7b^{-7} - 3a^6b^{-6} - 3a^5b^{-5} - 3ab^{-1} & 4.92\\ 
-3a^5b^{-5} - 3b^5a^{-5} & 4.99,\ 4.108 \\
-3a^2b^{-2} - 3ab^{-1} + 3b^2a^{-2} - 3b^5a^{-5} & 4.104\\
3a^9b^{-9} - 3a^5b^{-5} - 3a^3b^{-3} - 3ab^{-1} & 4.105
\end{array}
\]

\end{example}

\section{Questions}\label{Q}

We end with several observations and questions for future research.

First and foremost, the permutation Jones polynomials are a type of biquandle 
virtual bracket (see \cite{NOSY}). There are examples of other biquandles 
for which the pattern of skein coefficients from permutation Jones polynomials
also yields a valid biquandle virtual bracket. For example, the non-constant 
action biquandle given by the operation tables
\[
\begin{array}{r|rrr}
\utr & 1 & 2 & 3 \\ \hline
1 & 2 & 2 & 1 \\
2 & 1 & 1 & 2 \\
3 & 3 & 3 & 3
\end{array}
\quad
\begin{array}{r|rrr}
\utr & 1 & 2 & 3 \\ \hline
1 & 2 & 2 & 1 \\
2 & 1 & 1 & 2 \\
3 & 3 & 3 & 3
\end{array}
\]
has biquandle virtual bracket 
\[
\left[
\begin{array}{ccc|ccc|ccc}
a & 0 & 0 & b & 0 & 0 & 0 & v & v \\
0 & a & 0 & 0 & b & 0 & v & 0 & v \\
0 & 0 & a & 0 & 0 & b & v & v & 0 \\ \hline
a^{-1} & 0 & 0 & b^{-1} & 0 & 0 & 0 & v^{-1} & v^{-1} \\
0 & a^{-1} & 0 & 0 & b^{-1} & 0 & v^{-1} & 0 & v^{-1} \\
0 & 0 & a^{-1} & 0 & 0 & b^{-1} & v^{-1} & v^{-1} & 0
\end{array}\right].\]
How can we characterize and classify biquandles for which the permutation
Jones polynomial skein coefficients satisfy the biquandle virtual bracket 
axioms? Which properties of the permutation Jones poloynomials hold for these
other biquandle virtual brackets? 

Relatedly, the multiset of state-sum values over the set of $\sigma$-colorings
itself is an invariant; summing these to obtain a single polynomial may
lose some information. That is, there may be cases where two links have
equal $J_{\sigma}$ values but different state-sum multisets which sum to the
same overall polynomial. Are there such cases among classical links, virtual 
knots or virtual links? 

Are there Jones-equivalent classical links which are distinguished from
each other by $J_{\sigma}$ for some nontrivial permutation $\sigma$? For 
classical knots, the answer is no, but for virtual knots and links the 
answer is yes. No such example for classical links is found in our table
of small crossing number links for small $\sigma$, but there are infinitely
many $J_{\sigma}$s and infinitely many classical links we have yet to check.

Experimental computation shows that in many cases the permutation
Jones polynomials agree for a wide variety of permutations of the same
set $X$ for certain knots and links. What properties of two permutations 
$\sigma_1,\sigma_2:X\to X$ (e.g., cycle structure etc.) ensure that 
$J_{\sigma_1}=J_{\sigma_2}$ for all knots or links of certain types?

Speaking of computation, for $|X|>1$ the existence of multiple colorings 
means more total states to compute; however, each coloring's state-sum has 
only $2^{mc}\le 2^{cn}$ states where $mc$ is the number of monochromatic
colorings and $cn$ is the total crossing number, and so generally contains 
fewer states than the standard Jones polynomial. Moreover, for some 
permutations the contributions from different colorings may predictably 
coincide as in Example \ref{ex:2}, enabling faster computation by exploiting 
symmetries. Efficient computational algorithms are of interest here as always.

Finally, what enhancements and categorifcation of these invariants
can be defined and how are they related to each other and to other invariants?

\bibliography{sn-solo26}{}
\bibliographystyle{abbrv}

\bigskip

\noindent
\textsc{Department of Mathematical Sciences \\
Claremont McKenna College \\
850 Columbia Ave. \\
Claremont,  CA 91711}

\end{document}